\documentclass[a4paper,10pt,twoside]{amsart}
\usepackage{amsthm}
\usepackage{amsmath}
\usepackage{amssymb}
\usepackage{amsfonts}
\usepackage{amscd}
\usepackage[english]{babel}
\usepackage{stmaryrd}
\usepackage{enumerate}

\newtheorem{thm}{Theorem}
\newtheorem{cor}{Corollary}
\newtheorem{prop}{Proposition}
\newtheorem{rem}{Remark}

\newtheorem{lem}{Lemma}

\newtheorem*{thm2}{Theorem}

\newtheorem*{exam2}{Example}

\begin{document}

    \title{Common invariant subspace and commuting matrices}
    \author{Gerald BOURGEOIS}
    
    \address{G\'erald Bourgeois, GAATI, Universit\'e de la polyn\'esie fran\c caise, BP 6570, 98702 FAA'A, Tahiti, Polyn\'esie Fran\c caise.}
    \email{bourgeois.gerald@gmail.com}
        

\begin{abstract}
 Let $K$ be a perfect field, $L$ be an extension field of $K$ and $A,B\in\mathcal{M}_n(K)$. If $A$ has $n$ distinct eigenvalues in $L$ that are explicitly known, then we can check if $A,B$ are simultaneously triangularizable over $L$.
 Now we assume that $A,B$ have a common invariant proper vector subspace of dimension $k$ over an extension field of $K$ and that $\chi_A$, the characteristic polynomial of $A$, is irreducible over $K$.  Let $G$ be the Galois group of $\chi_A$. We show the following results
\begin{enumerate}
\item[(i)] If $k\in\{1,n-1\}$, then $A,B$ commute. 
\item[(ii)] If $1\leq k\leq n-1$ and $G=\mathcal{S}_n$ or $G=\mathcal{A}_n$, then $AB=BA$.
\item[(iii)] If $1\leq k\leq n-1$ and $n$ is a prime number, then $AB=BA$.
\end{enumerate}
Yet, when $n=4,k=2$, we show that $A,B$ do not necessarily commute if $G$ is not $\mathcal{S}_4$ or $\mathcal{A}_4$. Finally we apply the previous results to solving a matrix equation.
 \end{abstract}

\maketitle

    \section{Introduction}

  Throughout this paper, $K$ denotes a perfect field, and $\overline{K}$ an algebraic closure of $K$. Recall that a field $K$ is said to be perfect if every irreducible polynomial over $K$ has only simple roots in $\overline{K}$.
  
    \smallskip
    
    For $M\in\mathcal{M}_n(K)$, the set of $n\times{n}$ matrices with entries in $K$, $\sigma{(}M)$ denotes its spectrum, that is the set of its eigenvalues in $\overline{K}$. Two matrices $A,B\in\mathcal{M}_n(K)$ are said to be simultaneously triangularizable (denoted by ST) over $K$ if there exists a matrix $P\in{G}L_n(K)$ such that $P^{-1}AP$ and $P^{-1}BP$ are upper triangular. Thus such matrices have common invariant subspaces that form a complete flag over $K$. Note that if $A,B\in\mathcal{M}_n(K)$ commute, then they are ST over an extension field of $K$. In the sequel, $L$ denotes an extension field of $K$.
    
    \medskip  
    
 In Section 2, we consider $A,B\in\mathcal{M}_n(K)$ and we assume that $A$ has $n$ distinct eigenvalues in $L$, an extension field of $K$, and that $\sigma(A)$ is explicitly known. We give an algorithm which allows to check whether or not $A$ and $B$ are ST over $L$. Moreover, when $A$ and $B$ are ST we obtain a basis of $L$ that diagonalizes $A$ and triangularizes $B$. \\
 \indent  In Section 3, we assume that $A,B\in\mathcal{M}_n(K)$ have a common invariant proper vector subspace of dimension $k$ over $L$. We recall some criteria for the existence of common invariant proper subspaces of matrices. Shemesh gives this efficient criterion, when $k=1$, in \cite{1}
 \begin{thm2} Let $A,B\in\mathcal{M}_n(\mathbb{C})$. Then $A$ and $B$ have a common eigenvector if and only if 
 $$\bigcap_{p,q=1}^{n-1}\ker([A^p,B^q])\not=\{0\}.$$
  \end{thm2} 
  Note that the complexity of this test is in $O(n^5)$.\\
   When $k\geq 2$, a particular case, that is sufficient for our purpose, is treated in \cite{7},\cite{6} as follows. If $U\in\mathcal{M}_n(K)$, then $U^{(k)}$ denotes its $k^{th}$ compound. 
 \begin{thm} \label{tsatso}
Let $A,B\in \mathcal{M}_n(\mathbb{C})$, $k\in\llbracket{2},n-1\rrbracket$ be such that $A$ has distinct eigenvalues.
The following are equivalent
\begin{enumerate}
\item[(i)] $A,B$ have a common invariant subspace $W$ of dimension $k$.
\item[(ii)]  There exists $s\in\mathbb{C}$ such that $(A+sI_n)^{(k)}$ has distinct eigenvalues and $(A+sI_n)^{(k)},(B+sI_n)^{(k)}$ are invertible and have a common eigenvector in $\mathbb{C}^{\binom{n}{k}}$. Moreover this eigenvector is decomposable in the exterior product of $k$ vectors that constitute a basis of $W$.
\end{enumerate}
 \end{thm}
 We can show that the complexity of this test is at most (when $k=n/2$) in $O(\dfrac{2^{5n}}{n^{5/2}})$.
 \begin{rem}  The previous two results are also valid over a field that is algebraically closed and has characteristic $0$. 
 \end{rem}
 
   \medskip
 
  In the sequel, we work on a perfect field $K$ and we will use the following notation
 
 \medskip
  
\textbf{Notation}. Let $P\in K[x]$ be an irreducible polynomial of degree $n$. The splitting field $S_P$ of $P$ is $K(x_1,\cdots,x_n)$ where $x_1,\cdots,x_n$ are the roots of $P$ in $\overline{K}$.
 The Galois group of $P$ is the set of the $K$-automorphisms of $S_P$, that is 
 $$Gal(S_P/K)=\{\tau\in Aut(S_P)\;|\;\forall{t}\in{K},\tau(t)=t\};$$
  it is isomorphic to a subgroup of $\mathcal{S}_n$, the group of all the permutations of $\{1,\cdots,n\}$. If $M\in\mathcal{M}_n(K)$ and $\chi_M$, the characteristic polynomial of $M$, is irreducible, then $G_M$ denotes the Galois group of $\chi_M$.
 
\medskip

  \indent Assume that $A,B\in\mathcal{M}_n(K)$ have a common invariant proper subspace of dimension $k$ over an extension field $L$ of $K$ and that $\chi_A$ is irreducible over $K$. 
  We consider conditions that imply that $A,B$ commute. We show the following results.
\begin{enumerate}
\item[(i)] If $k\in\{1,n-1\}$, then $A,B$ commute. 
\item[(ii)] If $1\leq k\leq n-1$ and $G_A=\mathcal{S}_n$ or $G_A=\mathcal{A}_n$, then $AB=BA$.
\item[(iii)] If $1\leq k\leq n-1$ and $n$ is a prime number, then $AB=BA$.
\end{enumerate}
    The idea is as follows : let $F=[u_1,\cdots,u_k]$ be a $A$-invariant vector space where the $(u_i)_{i\leq k}$ are eigenvectors of $A$ associated to the eigenvalues $E=(\alpha_i)_{i\leq k}\subset\sigma(A)$. We seek elements of $G_A$ so that their orbits contain elements of $E$ and elements of $\sigma(A)\setminus E$. We consider $Bu_1\in F$ and we show that it is colinear to $u_1$.\\
In Section 4, we consider the case when $n=4,k=2$ and we show that the conclusion of (ii) may be false if we drop the hypothesis $G_A=\mathcal{S}_4$ or $G_A=\mathcal{A}_4$.\\
In Section 5, we use (i) and the simultaneous triangularization to solving the matrix equation $AX-XA=X^{\alpha}$ in a particular case.
 
 \section{An algorithm checking ST property}
 \begin{prop} \label{diag} Let $A,B\in\mathcal{M}_n(K)$  that are ST over $L$, an extension field of $K$. We assume that $A$ has distinct eigenvalues over $L$. Then there exists $S\in{G}L_n(L)$ such that $S^{-1}AS$ is diagonal and $S^{-1}BS$ is upper triangular.
 \end{prop}
 \begin{proof}
  There exists $P\in{G}L_n(L)$ such that $P^{-1}AP=T,P^{-1}BP=U$ where $T$ and $U$ are upper triangular. Note that $\sigma(A)\subset L$. The principal minors of $T$ are diagonalizable over $L$. By induction, we can construct a $T$-eigenvectors basis of $L^n$ such that the associated change of basis matrix is a upper triangular matrix $Q\in GL_n(L)$. Let $S=PQ$. Then  $S^{-1}AS$ is diagonal and $S^{-1}BS=Q^{-1}UQ$ is upper triangular.  
  \end{proof}
  \begin{rem} \label{propor} We may replace each column of $S$ with a proportional column.
  \end{rem}
  The previous result leads to an algorithm to check whether two such matrices are ST or not. Its complexity is in $O(n^4)$.
    \begin{prop}
 Let $A,B\in\mathcal{M}_n(K)$. We assume that $A$ has $n$ distinct eigenvalues in $L$, an extension field of $K$, and that we know explicitly $\sigma{(}A)$. Then we can decide if whether or not $A$ and $B$ are ST over $L$. If $A$ and $B$ are ST over $L$, then we obtain explicitly a matrix $S\in{G}L_n(L)$ that diagonalizes $A$ and triangularizes $B$.  
 \end{prop}
 \begin{proof}
 Since $A$ has distinct eigenvalues in $L$, we can calculate, from $\sigma{(}A)$, a $A$-eigenvectors basis of $L^n$. Let $R$ be the associated matrix and $Z=R^{-1}BR=[z_{i,j}]$.
 \begin{itemize}
\item[Case 1.] The matrices $A,B$ are ST. According to the proof of Proposition \ref{diag} and Remark \ref{propor}, there exists a permutation matrix $D$ such if $S=RD$, then $S^{-1}BS=D^{-1}ZD$ is upper triangular.\\
We consider the following algorithm

\medskip

  \noindent $U:=\{1,\cdots,n\}$. \\
 For every $i\leq n$, \\
 \indent if the $i^{th}$ column of $Z$ is zero, then $\alpha_i:=n$ \\
 \indent else $\alpha_i:=n-\sup \{j\leq n\;|\;z[j,i]\not= 0\}$. \\
 For $r$ from $1$ to $n$, do\\
 \indent find $i_r$ such that $\alpha_{i_r}=\sup_U \alpha_i$.\\
 \indent if $\alpha_{i_r}< n-r$ then $BREAK$\\
  \indent $d_r:=i_r$\\
   \indent $U:=U\setminus\{i_r\}$.\\
   If $r=n$ then $OUTPUT:=(d_1,\cdots,d_n)$\\
   else $OUTPUT:=NULL$. 
 
 \medskip
 
 \noindent The output $(d_1,\cdots,d_n)$ constitutes a convenient permutation. 
\item[Case 2.] The matrices $A,B$ are not ST. The previous algorithm gives the output $NULL$.
\end{itemize}
 \end{proof}
 \begin{rem} When $A$ has multiple eigenvalues or $\sigma(A)$ is unknown, to find an efficient algorithm is hard. A finite rational algorithm which allows to check whether two given $n\times{n}$ complex matrices are ST is exposed in \cite[Theorem 6]{5}. The study of complexity of the presented algorithm is omitted in \cite{5} and, as the author shows in \cite{4}, this test is impractical for $n\geq{6}$.
\end{rem}
 
\section{Common invariant subspace and commutativity}

\begin{prop} \label{subs} Let $A\in\mathcal{M}_n(K)$ such that $A$ has $n$ distinct eigenvalues in an extension field $L$ of $K$ and let $Z=\{B\in\mathcal{M}_n(K)\;|\;A,B \text{ have a common eigenvector }$\\
$\text{in }L^n\}$. Then $Z$ is the union of $n$ subspaces of $\mathcal{M}_n(K)$, each of them containing the commutant of $A$.
\end{prop}
\begin{proof}
Let $\alpha\in\sigma{(}A)$, $L_{\alpha}=K[\alpha]$ and $[{{L}_{\alpha}}:K]=k_{\alpha}$. Let $u\in{{L}_{\alpha}}^n\setminus\{0\}$ be such that $Au=\alpha{u}$. If $B=[b_{i,j}]\in{Z}$, then the condition ``$Bu$ and $u$ are linearly dependent" can be written in the form of $n-1$ ${{L}_{\alpha}}$-linear conditions on the $(b_{i,j})_{i,j}$, that is $k_{\alpha}\times{(}n-1)$ $K$-linear conditions on the $(b_{i,j})_{i,j}$. Thus $B$ is in a $K$-vector space of dimension at least $n^2-k_{\alpha}(n-1)$ that contains the commutant of $A$. Finally $B$ goes through the union of $n$ such subspaces.
\end{proof}
\begin{rem} \label{invsp} One has several interesting properties when $\chi_A$ is irreducible over $K$\\
$i)$ The endomorphism $A$ has no invariant proper subspaces over $K$. \\
$ii)$ Since $K$ is a perfect field, $A$ has simple eigenvalues in $S_{\chi_A}$ and its commutant is $K[A]$ and has dimension $n$.\\
$iii)$ According to \cite[p. 51]{2} and $ii)$, any $A$-invariant subspace of dimension $k$ over $\overline{K}$ is spanned by $k$ $A$-eigenvectors.
\end{rem}
 From now on, we suppose that $A$ and $B$ have a proper common invariant subspace of dimension $k$ over an extension field of $K$.
\begin{thm}  \label{comeig}
 Let $n\geq{2}$. Let $A,B\in\mathcal{M}_n(K)$ be such that they have a common eigenvector over an extension field of $K$. We assume that the characteristic polynomial of $A$ is irreducible over $K$. Then $AB=BA$. 
 \end{thm}
 \begin{proof}    \label{commut}
 Let $u$ be a common eigenvector and put $Au=\alpha u,Bu=\beta u$. Recall that $G_A$ is a transitive group, that is, there exist $(\tau_i)_{i=1,\cdots,n-1}\in G_A$ such that $$\sigma(A)=\{\alpha,\tau_1(\alpha),\cdots,\tau_{n-1}(\alpha)\}.$$
 Moreover, $\tau_i(u)$ is defined componentwise  and $A(\tau_i(u))=\tau_i(Au)=\tau_i(\alpha)\tau_i(u)$. Finally $\{u,\tau_1(u),\cdots,\tau_{n-1}(u)\}$ is an associated basis of eigenvectors of $A$.
 Thus $B(\tau_i(u))=\tau_i(Bu)=\tau_i(\beta)\tau_i(u)$ and $AB=BA$. 
 \end{proof}
   We can slightly improve the previous result as follows.
    \begin{lem}   \label{transp}
 If $A,B\in\mathcal{M}_n(L)$ have a common invariant subspace of dimension $k$ over $L$, then $A^T$ and $B^T$ have a common invariant subspace of dimension $n-k$ over $L$.  
  \end{lem}
   \begin{proof}
  The common invariant subspace of dimension $k$, can be written $V=\{X\in{L}^n\;|\;\Lambda{X}=0\}$ where $\Lambda\in\mathcal{M}_{n-k,n}(L)$ has maximal rank $n-k$. Since $\ker(\Lambda)\subset\ker(\Lambda{A})$, there exists $Z\in\mathcal{M}_{n-k}(L)$ such that $\Lambda{A}=Z\Lambda$, that is $A^T\Lambda^T=\Lambda^TZ^T$. The $n-k$ columns of $\Lambda^T$ span a vector space of dimension $n-k$ that is invariant for $A^T$.  
  \end{proof}
  \begin{cor}     \label{hyper}
   Let $A,B\in\mathcal{M}_n(K)$ be such that they have a common invariant hyperplane over an extension field of $K$. We assume that the characteristic polynomial of $A$ is irreducible over $K$. Then $AB=BA$.   
  \end{cor}
  \begin{proof}
 According to Lemma \ref{transp}, $A^T$ and $B^T$ have a common eigenvector and by Theorem \ref{comeig}, $A^TB^T=B^TA^T$, that implies $AB=BA$.   
  \end{proof}
  Now we consider the case where $A$ and $B$ have a common invariant proper subspace of dimension $\geq{2}$. Recall that $\mathcal{A}_n$, the group of even permutations of $\{1,\cdots,n\}$, contains the cycles of odd length. 
  \begin{thm}   \label{invsub}
Let $n\geq 3$ and $A,B\in\mathcal{M}_n(K)$ be such that they have a common invariant proper vector subspace over an extension field of $K$. We assume that $\chi_A$ is irreducible over $K$ and $G_A=\mathcal{S}_n$ or $G_A=\mathcal{A}_n$. Then $AB=BA$.   
  \end{thm}
  \begin{proof} Since $\chi_A=\chi_{A^T}$ and according to Lemma \ref{transp}, we may change $k$ with $n-k$ and assume that $k\leq \dfrac{n}{2}$, that implies $k+2\leq n$.
Let $F$ be a common invariant subspace of dimension $k\geq 2$ for $A,B$. According to Remark \ref{invsp}. $iii)$, the subspace $F$ is generated by certain eigenvectors $u_1,\cdots,u_k$ of $A$ respectively associated to the pairwise distinct eigenvalues of $A$: $\alpha_1,\cdots,\alpha_k$. Let $\sigma(A)=\{\alpha_1,\cdots,\alpha_k,\cdots,\alpha_n\}$. 
There exists $\tau\in \mathcal{A}_n\subset G_A$, a cycle of length $r=k+1$ if $k$ is even (resp. $r=k+2$ if $k$ is odd) such that, for every $1\leq i\leq r-1$, $\alpha_{i+1}=\tau(\alpha_i)$. Note that $F=[u_1,\cdots,\tau^{k-1}(u_1)]$ and $\tau^k(u_1)\notin F$.
Assume that $Bu_1=\sum_{i=0}^q\lambda_i\tau^i(u_1)$ where $q\in\llbracket 1,k-1\rrbracket$, for every $i$, $\lambda_i\in S_{\chi_{A}}$ and $\lambda_q\not= 0$. Therefore 
$$Bu_{k-q+1}=B(\tau^{k-q}(u_1))=\sum_{i=0}^{q-1}\tau^{k-q}(\lambda_i)\tau^{k-q+i}(u_1)+\lambda_q\tau^k(u_1)\in F.$$
 Then $\lambda_q=0$, that is a contradiction. Finally $Bu_1=\lambda_0u_1$ and we conclude by Theorem \ref{comeig}.
\end{proof}
    We can wonder if we still get the same conclusion of Theorem \ref{invsub} when droping the hypothesis $G_A=\mathcal{S}_n$ or $G_A=\mathcal{A}_n$.
  The answer is no in general but is yes if $n$ is a prime. 
     \begin{thm}
 Assume that $n\geq{3}$ is a prime number and let $A,B\in\mathcal{M}_n(K)$ be such that $\chi_A$ is irreducible over $K$. If $A$ and $B$ have a proper common invariant subspace, then $AB=BA$.  
  \end {thm}
  \begin{proof}  Let $F$ be a common invariant subspace of dimension $k\in\llbracket{2},n-1\rrbracket$ for $A,B$. Let $u\in F$ be an eigenvector of $A$ associated to $\alpha\in\sigma(A)$. Note that $n$ divides the cardinality of $G_A$. Since $n$ is prime and according to Cauchy's theorem, there exist $\tau\in G_A$ of order $n$. Necessarily the permutation $\tau$ is a cycle of length $n$ and $\sigma(A)=\{\alpha,\cdots,\tau^{n-1}(\alpha)\}$. 
 Moreover $\{u,\cdots,\tau^{n-1}(u)\}$ is a basis of eigenvectors of $A$ and some among these vectors constitute a basis of $F$. Put $Bu=\lambda_0u+\sum_{0<i\leq n-1}\lambda_i\tau^i(u)$ where the $(\lambda_i)_i$ are in $\overline{K}$. 
 Assume that there exists $p\in \llbracket{1},n-1\rrbracket$ such that $\lambda_p\not=0$. Since $n$ is prime and $k<n$, there exists an integer $q$ such that $\tau^q(u)\in F$ and $\tau^{q+p}(u)\notin F$. Therefore
 $$B(\tau^q(u))=\tau^q(\lambda_0)\tau^{q}(u)+\sum_{0<i<n,i\not= p}\tau^q(\lambda_i)\tau^{q+i}(u)+\tau^q(\lambda_p)\tau^{q+p}(u)\in F.$$
 Thus $B(\tau^q(u))$ is written as a linear combination of the basis $\{u,\cdots,\tau^{n-1}(u)\}$ and the coefficients of the vectors that are not in $F$ are zero.
 Consequently $\lambda_p=0$, that is a contradiction. Finally $Bu=\lambda_0 u$ and we conclude by Theorem \ref{comeig}. 
  \end{proof}
  \begin{rem}
  Consider $A,B\in\mathcal{M}_{35}(\mathbb{Q})$ such that $AB\not= BA$ (the verification is easy) and $G_A=\mathcal{S}_{35}$ or $\mathcal{A}_{35}$ (the verification is easy with the ``Magma'' software). Then, by Theorem \ref{invsub}, we deduce that $A,B$ admit no common invariant proper subspaces (the direct verification is impossible because the algorithm associated to Theorem \ref{tsatso} is impractical for $n>12$).
  \end{rem}
  \section{The case $n=4$}
  Assume that $A,B\in\mathcal{M}_4(K)$ have a common invariant subspace of dimension $k\in\{1,2,3\}$ and that $\chi_A$ is irreducible over $K$. If $k=1,3$, then from Theorem \ref{comeig} and Corollary \ref{hyper}, $AB=BA$. From Theorem \ref{invsub}, we obtain the same conclusion if $k=2$ and $G_A=\mathcal{S}_4$ or $\mathcal{A}_4$. It remains to study the cases where $A$ admits an invariant plane $\Pi$ and $G_A=\mathcal{C}_4$, the cyclic group with four elements, ${\mathcal{C}_2}^2$ or $\mathcal{D}_4$, the dihedral group with eight elements. Of course, if $K$ is a finite field, then necessarily $G_A=\mathcal{C}_4$.
   
   \medskip
   
   Let $A\in\mathcal{M}_n(K)$ be such that $\chi_A$ is irreducible over $K$ and $\Pi$ be a $A$-invariant plane. We denote by $r_A(\Pi)$ the dimension of the $K$-vector space of the matrices $B\in\mathcal{M}_n(K)$ such that $\Pi$ is a $B$-invariant plane.
  We will see that $r_A(\Pi)$ does not depend only on $k$ and $G_A$.
  
  \medskip
  \begin{prop} \label{diedral}
  Let $A\in\mathcal{M}_4(K)$ such that $\chi_A$ is irreducible, $G_A=\mathcal{D}_4$ and $\Pi$ is a $A$-invariant plane. Then $r_A(\Pi)=4$ or $8$.
  \end{prop}
  \begin{proof}
  There exist $\alpha_1,\alpha_2\in\sigma(A)$ such that $\Pi=\ker((A-\alpha_1 I_4)(A-\alpha_2 I_4))$.  Let $L=K(\alpha_1,\alpha_2)$. Note that $H=\{\tau\in G_A\;|\;\tau(\alpha_1)=\alpha_2\}$ has two elements. Let $u$ be an eigenvector of $A$ associated to $\alpha_1$. If $\tau\in H$, then $\{u,\tau(u)\}$ is a basis of $\Pi$.  
  Let $B\in\mathcal{M}_4(K)$ such that $\Pi$ is $B$-invariant. Therefore $Bu=\lambda u+\mu \tau(u)$ where $\lambda,\mu\in L$.
  \begin{itemize}
  \item Case 1. One element $\tau$ of $H$ has order $4$. Then $\sigma(A)=(\tau^i(\alpha_1))_{0\leq i\leq 3}$ and $(\tau^i(u))_{0\leq i\leq 3}$ is a basis of $K^n$ constituted by eigenvectors of $A$. Thus 
  $$B(\tau(u))=\tau(\lambda)\tau(u)+\tau(\mu)\tau^2(u)\in\Pi,$$
   that implies $\mu=0$. Therefore, for every $i$, $B(\tau^i(u))=\tau^i(\lambda)\tau^i(u)$ and $AB=BA$.
  \item Case 2. The elements of $H$ have order $2$. Then 
  $$B(\tau(u))=\tau(\lambda)\tau(u)+\tau(\mu)u\in\Pi.$$
   Let $(\tau_i)_{i=3,4}\in G_A$ such that $\tau_i(\alpha_1)=\alpha_i$. Clearly, $\{u,\tau(u),\tau_3(u),\tau_4(u)\}$ is a basis of eigenvectors of $A$ and, for $i=3,4$, $B(\tau_i(u))=\tau_i(Bu)$ depends only on $Bu$. Finally $B$ depends only on $\lambda,\mu\in L$ and $r_A(\Pi)=2[L:K]$. Necessarily $r_A(\Pi)<16$ and $[L:K]=4$ or $8$. Therefore $r_A(\Pi)=8$.  
  \end{itemize}
  \end{proof}
  \begin{prop}  \label{C4}
  Let $A\in\mathcal{M}_4(K)$ such that $\chi_A$ is irreducible, $G_A=\mathcal{C}_4$ and $\Pi$ is a $A$-invariant plane. Then $r_A(\Pi)=4$ or $8$.
  \end{prop}
  \begin{proof}
 We use the notations of Proposition \ref{diedral}. Here $[L:K]=4$ and $H$ has a unique element $\tau$.\\
 \begin{itemize}
 \item Case 1. $\tau$ is a generator of $G_A$. As in the proof of Proposition \ref{diedral}, Case 1, we show that $AB=BA$.
 \item Case 2. $\tau$ has order $2$. As in the proof of Proposition \ref{diedral}, Case 2, we show that $r_A(\Pi)=2[L:K]=8$.  
 \end{itemize}
  \end{proof}
 \begin{prop}
  Let $A\in\mathcal{M}_4(K)$ such that $\chi_A$ is irreducible, $G_A={\mathcal{C}_2}^2$ and $\Pi$ is a $A$-invariant plane. Then $r_A(\Pi)=8$.
  \end{prop}
  \begin{proof}
 Again we use the notations of Proposition \ref{diedral}. Here $[L:K]=4$, $H$ has a unique element $\tau$ and $\tau$ has order $2$.
 As in the proof of Proposition \ref{diedral}, Case 2, we show that $r_A(\Pi)=2[L:K]=8$.  
  \end{proof} 
\begin{exam2}{}~
\begin{itemize}
\item  We consider the following instance where $K=\mathbb{Q}$, $\chi_A(x)=x^4+x^3+x^2+x+1$ and $\Pi_{\epsilon}=\ker(A^2+\dfrac{1+\epsilon\sqrt{5}}{2}A+I_4)$ where $\epsilon=\pm 1$. Here $G_A=\mathcal{C}_4$, the element of $H$ has order $2$ and, according to Proposition \ref{C4}, $r_A(\Pi_{\epsilon})=8$. In particular, the following pair $(A,B)$ is such that the planes $\Pi_{\epsilon}$ are invariant for $A,B$ and yet, $A$ and $B$ are not ST.
 $$\;\;\;\;\;\;\;\;\;\;\;A=\begin{pmatrix}0&0&0&-1\\1&0&0&-1\\0&1&0&-1\\0&0&1&-1\end{pmatrix},B=\begin{pmatrix}0&-1&0&2\\-1&-1&1&1\\0&0&0&1\\1&0&0&0\end{pmatrix}\text{ where } G_B=\mathcal{D}_4.$$
 With the help of Theorem \ref{tsatso}, we show that $A,B$ admits only the planes $\Pi_{\epsilon}$ as proper common invariant subspaces over $\mathbb{C}$.\\
 $i)$ Applying the Shemesh's criterion to the couples $(A,B)$ and $(A^T,B^T)$, we conclude that there are no solutions in dimensions $1$ or $3$.\\
$ii)$ We prove easily that $(A+I_4)^{(2)}$ and $(B+I_4)^{(2)}$ have two common eigenvectors 
$$u_{\epsilon}=[1,\dfrac{-\epsilon\sqrt{5}+1}{2},1,\dfrac{-\epsilon\sqrt{5}+1}{2},\dfrac{-\epsilon\sqrt{5}+1}{2},1]^T.$$
An easy calculation shows that $u_{\epsilon}$ is the exterior product of the vectors of a basis of $\Pi_{\epsilon}$.\\ 
 \item Now we assume $K=\mathbb{Z}/7\mathbb{Z}$ and $\chi_A$ is as above. Then $\chi_A$ is irreducible and since $K$ is finite, $G_A=\mathcal{C}_4$. Moreover $5$ is not a square in $K$ and we can define $\sqrt{5}$ in an extension field of $K$ of dimension $2$. Then the planes $\Pi_{\epsilon}$, as above, are $A$-invariant. By the previous reasoning, we obtain $r_A(\Pi_{\epsilon})=8$. The matrices $A,B$, as above, admit the planes $\Pi_{\epsilon}$ as common invariant subspaces. We can show that $A,B$ have no other proper common invariant subspaces. Note that $\chi_B(x)=(x^2-x+4)(x^2+2x+2)$ and $B$ admits two invariant planes over $K$.
 \end{itemize} 
\end{exam2}
 \section{Solving a matrix equation}
 We give an application of Section 3 using the following known result
 \begin{thm2} 
   (McCoy's theorem \cite{3}) Let $L$ be an algebraically closed field and $A,B\in\mathcal{M}_n(L)$. Then $A,B$ are ST over $L$ if and only if for any polynomial $p(\lambda,\mu)$ in non-commuting indeterminates, $p(A,B)(AB-BA)$ is nilpotent.
   \end{thm2}  
    \begin{prop} \label{ST}
  Let $A=\begin{pmatrix}U&0_{p,q}\\0_{q,p}&V\end{pmatrix}\in\mathcal{M}_{p+q}(K)$ be such that $(U,V)\in \mathcal{M}_{p}(K)\times\mathcal{M}_{q}(K)$, $\chi_U$ and $\chi_V$ are  distinct irreducible polynomials over $K$. If $B\in\mathcal{M}_{p+q}(K)$, then $A,B$ are ST over $\overline{K}$ if and only if $B$ is in the form
  $$B=\begin{pmatrix}f(U)&Q\\0_{q,p}&g(V)\end{pmatrix},\text{ respectively }B=\begin{pmatrix}f(U)&0_{p,q}\\Q&g(V)\end{pmatrix}$$
   where $Q\in\mathcal{M}_{p,q}(K),\text{ respectively }Q\in\mathcal{M}_{q,p}(K)$ and $f,g\in K[x]$ are arbitrary.    
    \end{prop}
    \begin{proof}
   ($\Rightarrow$) Clearly $\sigma(A)=\{\sigma(U),\sigma(V)\}$ and $A$ has $p+q$ distinct eigenvalues. The eigenvectors of $A$ are in the form $[u,0]^T$ where ${u}^T$ is an eigenvector of $U$ or $[0,v]^T$ where ${v}^T$ is an eigenvector of $V$. Note that $A,B$ have a common eigenvector and assume, for instance, that it is in the form $[u,0]^T$ with $Uu^T=\alpha u^T$. We adapt the proof of Theorem \ref{comeig}: there exist $(\tau_i)_{i=1\cdots p}\in G_U$ such that $\sigma(U)=\{\alpha,\tau_1(\alpha),\cdots,\tau_{p-1}(\alpha)\}$. We deduce that   
    the $([\tau_i(u),0]^T)_{i\leq p}$ are eigenvectors of $B$ and $B$ is in the form $B=\begin{pmatrix}P&Q\\0_{q,p}&R\end{pmatrix}$ where $UP=PU$, that is $P=f(U)\in K[U]$. Then $AB-BA=\begin{pmatrix}0_{p}&UQ-QV\\0_{q,p}&VR-RV\end{pmatrix}$ and, more generally, $$p(A,B)(AB-BA)=\begin{pmatrix}0_{p}&*\\0_{q,p}&p(V,R)(VR-RV)\end{pmatrix}$$
   where $p(\lambda,\mu)$ is any polynomial in non-commuting indeterminates $\lambda,\mu$. According to the McCoy's theorem, $A,B$ are ST over $\overline{K}$ if and only for any polynomial $p(\lambda,\mu)$ in non-commuting indeterminates, $p(V,R)(VR-RV)$ is nilpotent, that is equivalent to : $V,R$ are ST. Then $V,R$ have a common eigenvector and, according to Theorem \ref{comeig}, 
   $VR=RV$, that is $R$ is a polynomial in $V$. \\
      ($\Leftarrow$) Again using the McCoy's theorem, the converse is clear.    
    \end{proof}
         Finally we apply Proposition \ref{ST} to solving a matrix equation.
    \begin{prop}
  Let $p,q$ be distinct positive integers. Let $A\in\mathcal{M}_{p+q}(K)$ be such that $\chi_A=\Phi\Psi$ where $\Phi$ and $\Psi$ are polynomials of degree $p$ and $q$, irreducible over $K$.
   Let $\alpha$ be a positive integer. Then the equation, in the unknown $X\in\mathcal{M}_{p+q}(K)$, 
  \begin{equation} \label{shap} AX-XA=X^{\alpha} \end{equation}
    admits the unique solution $X=0$.
       \end{prop}
    \begin{proof}  We may assume that $A=\begin{pmatrix}U&0_{p,q}\\0_{q,p}&V\end{pmatrix}$ where $U,V$ are the companion matrices of $\Phi,\Psi$.
    Since $X$ satisfies Equation (\ref{shap}), $A$ and $X$ are ST over $\overline{K}$ (cf. \cite{8}). According to Proposition \ref{ST}, necessarily $X$ has two possible forms, for instance this one
    $$X=\begin{pmatrix}f(U)&Q\\0_{p,q}&g(V)\end{pmatrix}\text{ and consequently }AX-XA=\begin{pmatrix}0_{p}&UQ-QV\\0_{q,p}&0_{q}\end{pmatrix}.$$
   $i)$ Assume $\alpha=1$. Equation (\ref{shap}) reduces to 
    $$f(U)=0\;,\;g(V)=0\;,\;UQ-QV=Q.$$
    The last equation can be rewritten $\phi(Q)=Q$ where $\phi=U\otimes I_q-I_p\otimes V^T$ is the sum of two linear functions that commute.
       Therefore 
   $$\sigma(\phi)=\{\lambda-\mu\;|\;\lambda\in\sigma(U),\mu\in\sigma(V)\}.$$
   If there are non-zero solutions, then there exist $\lambda\in\sigma(U),\mu\in\sigma(V)$ such that $\lambda-\mu=1$.
Since $\chi_U$ is the minimum polynomial of $\lambda$ over $K$, then  $\chi_U(x+1)$ is the minimum polynomial of $\mu$ over $K$ and $\chi_U(x+1)=\chi_V(x)$. That implies $p=q$, a contradiction.\\
$ii)$ Assume $\alpha>1$. Equation (\ref{shap}) reduces to 
    $$f(U)^{\alpha}=0\;,\;g(V)^{\alpha}=0\;,\;UQ-QV=0.$$  
    Let $\sigma(U)=(\lambda_i)_{i\leq p}$. Then $(f(\lambda_i))_{i\leq p}=\sigma(f(U))=\{0\}$. Since $f$ is a unitary polynomial of degree $p$, $f(x)=(x-\lambda_1)\cdots(x-\lambda_p)$. By Cayley-Hamilton Theorem, $f(U)=0$ and, in the same way, $g(V)=0$. By the reasonning used in $i)$, for every $\lambda\in\sigma(A),\mu\in\sigma(B)$, $\lambda-\mu\not= 0$ and $\phi$ is a linear bijection. We conclude that $Q=0$.
    \end{proof}

\medskip
 
 \textbf{Acknowledgements.}
 The author thanks David Adam and Roger Oyono for many valuable discussions. The author thanks the referee for helpful comments.

\bibliographystyle{plain}

\end{document}